\documentclass[12pt,reqno]{amsart}
\usepackage{amssymb}

\usepackage[curve,matrix,arrow]{xy}

\usepackage{pdfsync}

\baselineskip

\theoremstyle{plain} 
\numberwithin{equation}{section}
\newtheorem{thm}[equation]{Theorem}
\newtheorem{lemma}[equation]{Lemma}
\newtheorem{cor}[equation]{Corollary}
\newtheorem{prop}[equation]{Proposition}

\newtheorem{defi}[equation]{Definition}

\theoremstyle{definition}
\newtheorem{rem}[equation]{Remark}

\newtheorem{exm}[equation]{Example}
\newtheorem{construction}[equation]{Construction}

\theoremstyle{remark}

\def\fp{{\mathfrak{p}}}
\def\fq{{\mathfrak{q}}}
\def\fm{{\mathfrak{m}}}

\def\bN{{\mathbb N}}

\def\bZ{{\mathbb Z}}

\def\dim{\operatorname{dim}\nolimits}
\def\res{\operatorname{res}\nolimits}

\def\stmod{\operatorname{{\bf stmod}}\nolimits}
\def\D{\operatorname{{\bf D}}\nolimits}
\def\Db{\operatorname{{\bf D^{f}}}\nolimits}
\def\CatC{\operatorname{{\bf C}}\nolimits}

\def\dg{\operatorname{DG}\nolimits}

\def\thick{\operatorname{Thick}\nolimits}
\def\Supp{\operatorname{Supp}\nolimits}
\def\spec{\operatorname{Spec}^{*}\!}

\def\HHH{\operatorname{H}\nolimits}
\def\Hom{\operatorname{Hom}\nolimits}

\def\RHom{\operatorname{RHom}\nolimits}
\def\Ext{\operatorname{Ext}\nolimits}

\def\ann{\operatorname{Ann}\nolimits}
\def\cone{\operatorname{cone}\nolimits}
\def\bs{\boldsymbol}
\def\col{\colon}

\def\lotimes#1{{\otimes^{\mathsf L}_{#1}\,}}
\newcommand\loc[2]{{#1}_{{\scriptscriptstyle (}#2{\scriptscriptstyle )}}}
\def\shift{{\Sigma}}
\def\xra{\xrightarrow}

\newcommand\da[2]{{\downarrow^{#1}_{#2}}}
\newcommand\ua[2]{{\uparrow^{#2}_{#1}}}
\def\xla{\xleftarrow}
\def\lra{\longrightarrow}

\def\sfh{{\mathsf h}}
\def\sfi{{\mathsf i}}

\def\sfr{{\mathsf r}}
\def\sft{{\mathsf t}}

\title[Bounded derived category of a finite group]{Thick subcategories 
of the bounded \\ derived category of a finite group}
\author[Jon F. Carlson]{Jon F. Carlson}
\address{Department of Mathematics, University of Georgia, Athens, GA 30602, USA}
\email{jfc@math.uga.edu}

\author[Srikanth B. Iyengar]{Srikanth B. Iyengar}
\address{Department of Mathematics, University of Nebraska, Lincoln, NE 68588, USA}
\email{iyengar@unl.edu}

\thanks{Research partially supported by NSF grants DMS-1001102 (JFC) and DMS-0903493 (SBI)}

\date\today

\begin{document}

\begin{abstract} 
A new proof of the classification for tensor ideal thick subcategories of the bounded derived category, and the stable category, of modular representations of a finite group is obtained. The arguments apply more generally to yield a classification of thick subcategories of the bounded derived category of an artinian complete intersection ring. One of the salient features of this work is that it takes no recourse to infinite constructions, unlike the previous proofs of these results.
\end{abstract}

\maketitle

\section{Introduction}
In the paper \cite{BCR}, the first author, in collaboration with Benson and Rickard, established a classification for tensor ideal thick subcategories of the 
stable category of finitely generated $kG$-modules in the 
case that $G$ is a finite group and $k$ is a field of characteristic 
$p>0$ dividing the order of $G$. This result was inspired and
influenced by an earlier classification of the thick subcategories 
of the perfect complexes over a commutative noetherian ring 
by Hopkins \cite{Hopkins:1987}. The statements are remarkably 
similar: In both cases a subcategory is determined by the support, 
suitably defined, of the objects in the subcategory. 

However, the methods in the two settings were entirely different. 
At that time it did not seem possible to adapt Hopkins' methods 
to modules over group algebras. Instead, the proofs in \cite{BCR} 
used idempotent modules and idempotent  functors developed 
by Rickard.  Idempotent modules are, in general, infinitely 
generated, hence the proof had employed constructions from 
the stable category of all $kG$-module to obtain a result that 
spoke only of finitely generated module. 

The main purpose of this paper is to show that the results in the 
two settings are directly related. To this end, we extend Hopkins' 
arguments to obtain a classification of the thick subcategories 
of perfect Differential Graded modules over suitable DG algebras. 
The result about $kG$-modules is then deduced from it by a 
series of reductions, following the paradigm developed in the 
work of the second author with Avramov, Buchweitz, and 
Miller~\cite{ABIM}, and with Benson and Krause~\cite{BIK3}.  
What is more, the arguments require \emph{no} constructions 
involving infinitely generated modules, in contrast to the proof in \cite{BCR}.  

In the reduction mentioned above, it is more natural to work 
with the bounded derived category of modules over $kG$;  
the sought after classification for the stable category is an 
easy consequence, for the latter is a quotient of the former 
by the subcategory of perfect complexes. The first step (which comes
last in the paper) is to reduce to the case of an elementary abelian $p$-group; 
this uses a theorem of the first author~\cite{C}. 
After that, If $p=2$ Koszul 
duality gives an equivalence of triangulated categories between 
the bounded derived category of $k(\bZ/2)^{r}$-modules and 
the derived category of DG modules, with finitely generated 
cohomology, over a polynomial ring $k[y_{1},\dots,y_{r}]$ with 
generators $y_{i}$ in degree one. Then one can apply 
(the extension of) Hopkins' theorem to $k[\bs y]$ to get the 
desired classification. Even in this case, we take a slightly 
longer route, that has the merit of being independent 
of the characteristic of $k$.

The starting point is the observation that the group algebra of an elementary abelian group is 
an artinian complete intersection ring. For any such ring $R$ 
there is an exact functor, constructed in \cite{ABIM}, from its 
bounded derived category to the derived category of DG modules 
with finitely generated cohomology over a polynomial ring 
$k[x_{1},\dots,x_{r}]$, with generators $x_{i}$ in degree two. 
Though not an equivalence of categories, the functor is close 
enough to it that one can relate the thick subcategories in 
the source  and in the target. Hopkins' Theorem again applies, 
and  yields a classification for thick subcategories of the 
bounded derived category of $R$. We note that such a 
result has been proved for complete intersections of any 
(Krull) dimension by Stevenson~\cite{St}, using different 
techniques involving infinite methods.

Friedlander and Pevtsova~\cite{FP} have proved a similar 
classification of thick subcategories for the stable category 
of modules over a finite group scheme, again using 
idempotent modules. We have been unable to bring to bear 
the methods of this paper in that context. The difficulty is 
that for more general finite group schemes, there is no reduction of
such a classification to unipotent abelian subgroup 
schemes, which are the analogs of the elementary abelian subgroups

\medskip

Here is an outline of the contents of this paper: Section~\ref{sec:cdga} is devoted to establishing the analogue of Hopkins' theorem for DG modules over commutative  graded rings. This uses a notion of support for DG modules, discussed in Section~\ref{sec:support}. In Section~\ref{sec:lambda} we classify the thick subcategories of the bounded derived category of DG modules over an exterior algebra, using a BGG correspondence from \cite{ABIM} and the main result of Section~\ref{sec:cdga}. The statement  on artinian complete intersections is deduced from it, in Section~\ref{sec:ci}. Section~\ref{sec:kg} is devoted to group algebras. 

\medskip

In this article, we assume that the reader is familiar with basic results on group cohomology, and take \cite{CTVZ} as our reference for this topic. We also take the mechanics of triangulated categories as given, and refer the reader to one of the standard texts such as \cite{Hap} for details. We recall only that a non-empty full subcategory of a triangulated category is \emph{thick} if it is triangulated ({\it i.e.} if two of three objects in a triangle are in the subcategory, then so is the third) and closed under  summands. If $M$ is an object in a triangulated category, the smallest thick subcategory containing it is denoted $\thick(M)$. 

\medskip

There are two distinct aspects to any classification of thick subcategories via supports. The first is the \emph{realizability problem}: Construct an object with support a prescribed closed subset of the ambient variety.  This now well-understood to be the more elementary aspect of the classification problem, and there are various solutions that apply in great generality; see, for example, \cite{AI2}.  The second, more challenging, step is to prove that if the support of an object $M$ is contained in the support $N$, then $\thick(M)\subseteq\thick(N)$. This is where properties specific to the context at hand come in play.  In this work, the focus is on this aspect of the classification problem.

\subsubsection*{Acknowledgments}
This collaboration grew out of conversations at the Morningside Center for Mathematics, Beijing, during the workshop ``Representation Theory: Cohomology and Support'' held in September 2011. It is a pleasure to thank the Morningside Center for the hospitality, and the organizers: Bangming Deng, Henning Krause, Jie Xiao, Hechun Zhang, Bin Zhu, for the invitation to the workshop. 

%

\section{Support for DG modules}
\label{sec:support}
Let $S$ be a commutative $\bZ$-graded noetherian ring; we assume $S$ is concentrated in even degrees or $2S=0$. In this article, are gradings are upper; thus $S=\oplus_{i\in\bZ}S^{i}$. In this section we introduce a notion of support for Differential Graded (henceforth abbreviated to DG) modules over $S$ and describe various ways of computing it. 

To begin with, we record a version of Nakayama's Lemma for graded $S$-modules. Such results are known in greater generality  but the version below is all we need; the proof is similar to the one for ungraded rings, and is sketched for completeness.
 
\begin{lemma}
\label{lem:nak}
Assume $S$ has a unique maximal homogeneous ideal, say $\fm$. If $M$ is a finitely generated graded $S$-module and
$\fm M = M$, then $M=0$.
\end{lemma}

\begin{proof}
Pick homogeneous elements $\{x_{1},\dots,x_{n}\}$ that generate $M$ as an $S$-module. By hypothesis, there exist homogeneous elements $\{s_{ij}\}_{1\leqslant i,j\leqslant n}$ in $\fm$ such that
\[
x_{i} = \sum_{j=1}^{n}s_{ij}x_{j} \quad \text{for each $1\leq i\leq n$.} 
\]
By the determinant trick, it follows that $\det(I -A)M=0$ where $A=(s_{ij})$. Observe that $\det(I-A)= 1 - \sum_{t\in\bZ}s_{t}$, where $s_{t}\in \fm_{t}$. Thus, for any homogeneous element $x$ in $M$ one gets $x = \sum_{t}s_{t}x$, so that $x=s_{0}x$. As $S_{0}$ is a local ring with maximal ideal $\fm\cap S_{0}$, the element $1-s_{0}$, which is in $S_{0}\setminus \fm$, is a unit, and hence $x=0$.
\end{proof}

Henceforth we view $S$ as a DG algebra with $d^{S}=0$. Given a DG $S$-module $M$, its homology $\HHH(M)$ is  a graded $S$-module. We write $\D(S)$ for the derived category of DG $S$-modules, viewed as a triangulated category, and $\Db(S)$ for the full subcategory consisting of DG modules $M$ for which the $S$-module $\HHH(M)$ is finitely generated. The basic constructions and results on DG modules over DG algebra required in this work are all recorded in \cite[\S3 and \S4]{ABIM}. 

We recall that a DG module $S$-module $F$ is \emph{semi-free} if it admits a family $(F_{n})_{n\in\bN}$ of DG submodules such that $F_{n}\subseteq F_{n+1}$,  and $\cup_{n}F_{n}=F$, and $F_{n+1}/F_{n}$ is isomorphic to a direct sum of suspensions of $S$; in particular, forgetting differentials, $F$ is a graded free $S$-module. When $F$ is semi-free the functors $\Hom_{S}(F,-)$ and $-\otimes_{S}F$ defined on the category of DG $S$-modules preserve quasi-isomorphisms.

Each DG module $M$ over $S$ admits a \emph{semi-free resolution}: a quasi-isomorphism of DG $S$-modules $F\to M$ with $F$ semi-free. Then assignments
\[
(-\lotimes{S}M) = -\otimes_{S}F \quad\text{and}\quad \RHom_{S}(M,-) = \Hom_{S}(F,-)
\]
then define exact functors on $\D(S)$.

\medskip

We write $\spec S$ for the set of \emph{homogeneous} prime ideals in $S$. For each $\fp$ in $\spec S$, the homogeneous localization of a graded $S$-module $M$ at $\fp$ is denoted $\loc M{\fp}$. We write $k(\fp)$ for $\loc S{\fp}/\fp\loc S{\fp}$, which is a graded field. See Bruns and Herzog~\cite[\S1.5]{Bruns/Herzog:1998} for the module theory over graded rings. 

When $M$ is a DG $S$-module, and $\fp\in\spec S$, the differential on $M$ induces one on $\loc M{\fp}$, making it a DG module over $\loc S{\fp}$.

\begin{defi}
The support of $M\in\Db(S)$ is the subset of $\spec S$ defined by
\[
\Supp^{*}_{S}M = \{\fp \in \spec S\mid \HHH(\loc M{\fp})\ne 0\}
\]
\end{defi}
Let $\ann_{S}\HHH(M)$ denote the \emph{annihilator} of the $S$-module $\HHH(M)$; it is a homogeneous ideal of $S$.
Since $\HHH(\loc M{\fp})\cong \loc{\HHH(M)}{\fp}$, there are equalities
\begin{equation}
\label{eq:suppH}
\Supp^{*}_{S}M =\Supp^{*}_{S}\,\HHH(M) = \{\fp\in\spec S\mid \fp\supseteq \ann_{S}\HHH(M)\}\,. 
\end{equation}
The next result provides a description of support that is often more convenient to work with. Its proof  is, at the end, an application of Nakayama's Lemma~\ref{lem:nak}.

\begin{thm}
\label{thm:support}
Let $M$ be a DG $S$-module in $\Db(S)$. For each $\fp\in\spec S$ one has
\[
\HHH(\loc M{\fp})\ne 0 \quad \text{if and only if}\quad \HHH(k(\fp)\lotimes SM)\ne 0\,.
\]
In particular $\Supp^{*}_{S}M =\{\fp\in\spec S\mid k(\fp)\lotimes SM \ne 0 \text{ in } \D(S)\}$.
\end{thm}

\begin{proof}
Replacing $S$ by $\loc S{\fp}$ one may assume $S$ is (graded) local, so that there is a unique  maximal homogeneous ideal; call it $\fm$ and let $k=S/\fm$. The task is then to verify that, in $\D(S)$, if $k\lotimes SM=0$, then $M=0$. Pick a set of elements $\{s_{1},\dots,s_{n}\}$ that generate the ideal $\fm$ and let $K$ be the DG $S$-module obtained by an iterated mapping cone: $K_{0}=S$ and $K_{i}=\cone(\shift^{|s_{i}|} K_{i-1}\xra{\ s_{i}\ }K_{i-1})$ for $i\geq 1$ and $K=K_{n}$. Then $K$ is in $\thick_{S}(k)$; see the penultimate paragraph in the proof of \cite[Theorem 8.1]{BIK2}. The hypothesis thus implies that $K\otimes_{S}M=0$ in $\D(S)$. Consider the long exact sequence associated to the exact sequence of DG modules
\[
0\lra K_{i-1}\otimes_{S}M \lra K_{i}\otimes_{S}M \lra \shift^{|s_{i}|+1}(K_{i-1}\otimes_{S}M) \lra  0
\]
Assuming $\HHH(K_{i}\otimes M)=0$ there is an isomorphism 
\[
\shift^{|s_{i}|}\HHH(K_{i-1}\otimes_{S}M)\xra{\ s_{i}\ }\HHH(K_{i-1}\otimes_{S}M)
\]
of finitely generated graded $S$-modules. Hence  $\HHH(K_{i-1}\otimes_{S}M)=0$, by Lemma~\ref{lem:nak}, for $s_{i}$ is in $\fm$. An iteration then yields $\HHH(M)=0$, as desired.
\end{proof}

\section{Perfect DG modules over commutative graded algebras}
\label{sec:cdga}
Let $S$ be a commutative $\bZ$-graded noetherian ring, either concentrated in even degrees or satisfying $2S=0$.
As in the previous section, we view $S$ as a DG algebra with differential zero. The main result of this section is a generalization of Hopkins' theorem that applies to perfect DG modules over $S$. 

\begin{defi}
A DG $S$-module is said to be \emph{perfect} if it is in $\thick_{S}(S)$.
\end{defi}

A DG module is perfect if and only if it is a retract (that is to say, isomorphic in $\Db(S)$ to a direct summand of) a semi-free DG module with a finite free filtration; see \cite[Theorem~4.2]{ABIM} for details. We do not need this description in what follows.

When $S$ is concentrated in degree zero, a DG module is nothing other than a complex of $S$-modules, and a DG module is perfect if and only if it is isomorphic to a bounded complex of finitely generated projective $S$-modules. Thus the result below generalizes Hopkins' theorem~\cite[Theorem 11]{Hopkins:1987}, see also Neeman~\cite[Lemma 1.2]{Neeman:1992} and Thomason~\cite[Theorem 3.14]{Thomason:1997}, from rings (which are DG algebras concentrated in degree $0$) to DG algebras with zero differential. 

\begin{thm}
\label{thm:hopkins-dga}
If $M$ and $N$ are perfect DG $S$-modules with $\Supp^{*}_{S}M\subseteq \Supp^{*}_{S}N$, then $M$ is in $\thick_{S}(N)$.
\end{thm}

This result can be deduced from the classification of the localizing subcategories of $\D(S)$ proved in \cite[\S8]{BIK2}. In the interest of avoiding  ``infinite constructions'', we sketch an alternative proof, mimicking arguments from \cite{Hopkins:1987,Neeman:1992,Thomason:1997}. The crux is the following \lq Tensor Nilpotence Theorem\rq\ for perfect DG modules that extends \cite[Theorem 10]{Hopkins:1987} and \cite[Theorem 1.1]{Neeman:1992}, see also \cite[Theorems 3.6,3.8]{Thomason:1997}.

\begin{thm}
\label{thm:tnt}
Let $F$ and $G$ be perfect DG $S$-modules. If $f\col F\to X$ is a morphism of DG $S$-modules and $k(\fp)\lotimes Sf=0$ for each $\fp\in\Supp^{*}_{S}G$, then there exists an integer $n\ge 1$ such that $G\lotimes Sf^{\otimes n}=0$ in $\D(S)$.
\end{thm}

Here $f^{\otimes n}$ denotes the morphism $X^{\lotimes{S}\!n}\to Y^{\lotimes{S}\!n}$ induced by $f$.

\begin{proof} The proof proceeds by a series of reductions. 

\medskip

\emph{Step }1. Reduction to the case $G=S$.

\medskip

This is a formal argument, see \cite[Theorem 3.8]{Thomason:1997}, exploiting the fact that, since the DG $S$-module $G$ is perfect, it is a retract of $\RHom_{S}(G,G)\lotimes{S}G$.

\medskip

\emph{Step }2. Reduction to the case $F=S$.

\medskip

Indeed, $f^{\otimes n}=0$ if and only if $(f')^{\otimes n}=0$ where $f'\col S\to \RHom_{S}(F,X)$ is the morphism that assigns $1$ to $f$.

\medskip

For the sequel, we replace $X$ by its semi-free resolution; thus, for example, $W\lotimes Sf$ is represented by $W\otimes_{S}f$ for any $W\in\D(S)$ and $f^{\otimes n}\col S\to X^{\lotimes{S}n}$ is represented by the morphism $f^{\otimes n}\col S\to X^{\otimes n}$. In particular, the following conditions  are equivalent.
\begin{itemize}
\item $f^{\otimes n}=0$ in $\D(S)$;
\item the cycle $f(1)^{\otimes n}\in X^{\otimes n}$ is homologous to $0$.
\end{itemize} 

Note that morphisms $g,h\col S\to W$ are homotopic if and only if the cycles $g(1)$ and $h(1)$ in $W_{0}$ are homologous; in such a situation, we write $g(1)\sim h(1)$.

\medskip

\emph{Claim }1. Let $I,J$ be homogeneous ideals in $S$ such $(S/I)\otimes_{S}f=0=(S/J)\otimes_{S}f$.
If $IJ=(0)$, then $f\otimes_{S}f=0$.

\medskip

Indeed, the hypothesis is equivalent to the existence of cycles $a\in IX_{0}$ and $b\in JX_{0}$ such that $f(1)\sim a$ and $f(1)\sim b$. It is then readily seen that
\[
(f\otimes f)(1) = f(1) \otimes f(1) \sim a \otimes b\,.
\]
Since $a\otimes b \in IJ(X_{0}\otimes_{S}X_{0})=0$, it follows that $f\otimes f=0$.

\medskip

In what follows,  $\dim S$ denotes the dimension of $\spec S$.

\medskip

\emph{Claim }2. If $d$ is a non-negative integer such that the result holds for all graded domains $S$ with $\dim S\leq d$, then it holds for all graded rings $S$ with $\dim S\leq d$.

\medskip

Indeed, this is by a straightforward application of Claim 1, given  that the minimal primes of $S$, and so also its nil-radical, are homogeneous; see \cite[Lemma~1.5.6]{Bruns/Herzog:1998}.

\smallskip

Next we prove the desired result when $\dim S$ is finite, by an induction argument.

\medskip

Assume $\dim S=0$. Using the preceding claim we can assume $S$ is a domain, and hence a graded field. In this case the result is obvious.

We now suppose that the result holds for all commutative $\bZ$-graded noetherian rings of dimension $d-1$ or less, and verify that it holds whenever $\dim S=d\geq 1$; by Claim 2, it suffices to consider the case when $S$ is a domain. 

Consider the ideal $\ann(f) = \{s\in S\mid sf=0 \text{ in } \D(S)\}$ of $S$. By hypothesis, $k(0)\otimes_{S}f=0$, where $k(0)$ is the graded residue field at the prime ideal $(0)$. Thus, $\ann(f)\ne (0)$. We may assume $f\ne 0$, so there exists a non-zero homogeneous element $s\in \ann(f)$ that is not a unit, and hence a non-zero divisor in $S$. 

Consider the morphism $S/(s)\otimes_{S} f\col S/(s)\to S/(s)\otimes_{S} X$. Since $\spec (S/(s))$ is naturally identified with the subset of primes $\fp\in\spec S$ containing $s$, and for such a $\fp$ the $S$ action on $k(\fp)$ factors through $S/(s)$, it follows that 
\[
k(\fp)\otimes_{S/(s)} (S/(s)\otimes_{S} f) \cong k(\fp)\otimes_{S} f =0\,.
\]
The induction hypothesis thus yields that $S/(s)\otimes_{S}f^{\otimes n}=0$ for some integer $n\geq 1$; that is to say, $f^{\otimes n}(1)\sim sx$ for some $x\in X$. Because $sd(x) = d(sx)=0$ and the graded $S$-module underlying $X$ is free (recall $X$ is semi-free) one gets $d(x)=0$, that is to say, $x$ is a cycle. This justifies the second equivalence below:
\[
f^{\otimes(n+1)}(1)= (f^{\otimes n}\otimes f)(1) \sim (sx \otimes f(1)) = (x \otimes sf(1)) \sim (x\otimes 0) = 0\,.
\]
The other equivalence and the equalities are standard. Thus  $f^{\otimes(n+1)}=0$.

This completes the proof of the theorem when $\dim S$ is finite.

\medskip

To finish the proof, we consider the chain of ideals 
\[
\ann(f)\subseteq \ann(f^{\otimes 2})\subseteq \ann(f^{\otimes 3})\cdots\,.
\]
Since $S$ is noetherian, there exists an integer $n$ such that $\ann(f^{\otimes n})=\ann(f^{\otimes i})$ for all $i\geq n$. We claim that $f^{\otimes n}=0$, or equivalently, that $\ann(f^{\otimes n})= S$.

Indeed, if not then there exists a prime ideal $\fp\supseteq \ann(f^{\otimes n})$. The dimension of $\loc S{\fp}$ is finite and the hypothesis of the theorem apply to the morphism $\loc f{\fp}\col \loc S{\fp}\to \loc X{\fp}$, so the already established case yields that for ${\loc f{\fp}}^{\otimes i}=0$ for $i\gg 0$  so that
\[
\loc {\ann(f^{i})}{\fp} = \ann({\loc f{\fp}}^{\otimes i}) = \loc S{\fp}\,.
\]
This is a contradiction.
\end{proof}

In the sequel, given a morphism $g\col X\to Y$ of DG $S$-modules, we write $\cone(g)$ for any DG $S$-module that fits in an exact triangle $X\xra{\ g\ }Y \to\cone(g)\to $, and use the fact that if $g=0$, then $Y$ is a retract of $\cone(g)$.

\begin{proof}[Proof of Theorem~\emph{\ref{thm:hopkins-dga}}]
Given Theorem~\ref{thm:tnt} the argument is akin to the one for  \cite[Theorem 7]{Hopkins:1987}, see also \cite[Lemma 1.2]{Neeman:1992} and \cite[Lemma 3.14]{Thomason:1997}, so we only sketch it.

The natural morphism $S\to \RHom_{S}(N,N)$ of DG $S$-modules induces a morphism 
\[
g\col M\to M\lotimes{S}\RHom_{S}(N,N)\,,
\]
and this gives rise to an exact triangle of perfect DG $S$-modules:
\[
F\xra{\ f\ } M\xra{\ g\ } M\lotimes{S}\RHom_{S}(N,N)\lra
\]
For any $\fp\in\Supp^{*}_{S}M$, one has $\HHH(k(\fp)\lotimes{S}N)\ne 0$, since $\Supp^{*}_{S}M\subseteq\Supp^{*}_{S}N$; here we use the alternative definition of support from Theorem~\ref{thm:support}. Since $k(\fp)$ is a graded field, the map $k(\fp)\lotimes{S}g$ is split-injective, so that $k(\fp)\lotimes{S}f=0$ in $\D(k(\fp))$. Theorem~\ref{thm:tnt} applies to yield that $M\lotimes{S}f^{\otimes n}=0$ for some positive integer $n$. Thus, $M^{\otimes (n+1)}$ is a retract of $M\lotimes S\cone(f^{\otimes n})$. A straightforward induction on $i$ shows that $\cone(f^{\otimes i})$ is in $\thick_{S}(N)$ for each $i\geq 1$, and hence so is $M\lotimes{S}\cone(f^{\otimes i})$. In conclusion, $M^{\otimes (n+1)}$ is in $\thick_{S}(N)$, and it follows that so is $M$; this uses the fact that, since $M$ is perfect, it is a retract of the DG $S$-module $\RHom_{S}(M,M)\lotimes SM$, which is isomorphic to $\RHom_{S}(M,S)\lotimes{S}M^{\otimes 2}$.
\end{proof}

\begin{defi}
A subset $V$ of $\spec S$ is specialization closed when the following property holds: If $\fp\subseteq\fq$
are homogenous prime ideals in $S$ and $\fp\in V$, then $\fq\in V$; equivalently, $V$ is a union of closed subsets, in the Zariski topology on $\spec S$.
\end{defi}

In what follows, given a subset $V$ of $\spec S$, set
\[
\thick(S)_{V}\ = \ \{M\in\thick(S)\mid \Supp^{*}_{S}M\subseteq V\}\,,
\]
viewed as a full subcategory of $\thick(S)$. It is not hard to verify that this is a thick subcategory when  $V$ is specialization closed. Theorem~\ref{thm:hopkins-dga} yields a perfect converse:

\begin{cor} 
\label{cor:thick-cdga}
If $\CatC$ is a thick subcategory of\, $\thick(S)$, then there exist a specialization closed subset $V\subseteq \spec S$ such that $\CatC = \thick(S)_{V}$.
\end{cor}

\begin{proof}
Since the support of any DG $S$ module $M$ in $\Db(S)$ is a closed subset of $\spec S$, by \eqref{eq:suppH}, the subset
\[
V=\bigcup_{M\in\CatC} \Supp^{*}_{S}M
\]
is specialization closed. It is immediate from definitions that $\CatC\subseteq \thick(S)_{V}$.

Let now $M$ be a perfect DG $S$-module with $\Supp^{*}_{S}M\subseteq V$. Since $\Supp^{*}_{S}M$ is a closed subset, there exists DG $S$-modules $N_{1},\dots,N_{s}$ in $\CatC$ such that 
\[
\Supp^{*}_{S}M \subseteq \bigcup_{i}\Supp^{*}_{S}(N_{i}) 
\]
Note that the subset on the right is the support of the DG $S$-module $N=\oplus_{i}N_{i}$, that is also perfect. Thus Theorem~\ref{thm:hopkins-dga} yields that $M$ is in $\thick_{S}(N)$, and hence in $\CatC$, since the latter is a thick subcategory. This proves that $\thick(S)_{V} \subseteq \CatC $.
\end{proof}

The realizability problem in $\thick(S)$ is easily solved: For any closed subset $V$ of $\spec S$, defined by an  homogeneous ideal $I$ in $S$, the perfect DG $S$-module $S/\!\!/I$ constructed as in \cite[\S2.5]{BIK2} has support $V$.  In conjunction with Corollary~\ref{cor:thick-cdga}, this yields a bijection between thick subcategories of $\thick(S)$ and specialization closed subsets of $\spec S$.

\section{DG modules over exterior algebras.}
\label{sec:lambda}
In this section, we sketch the development of a theory of support varieties for DG modules over an exterior algebra, and prove the version of Hopkins' theorem in that context. A lot of this material is taken from \cite{ABIM,AI}. 

Suppose that $k$ is a field and that $\Lambda = k\langle z_1, \dots, z_r \rangle$ is an exterior algebra on indeterminates $\{z_{i}\}$, all of odd negative degrees; recall that we are using upper gradings, unlike in \cite{ABIM,AI}. We consider $\Lambda$ as a DG algebra with $d^{\Lambda}=0$, let $\D(\Lambda)$ denote the derived category DG of $\Lambda$-modules, and let $\Db(\Lambda)$ denote the full subcategory consisting of DG $\Lambda$-modules $M$ such that the $\Lambda$-module $\HHH(M)$ is  finitely generated.
We note that $\Db(\Lambda)$ coincides with $\thick_{\Lambda}(k)$, the thick subcategory generated by the trivial module $k$; see \cite[Remark~7.5]{ABIM}.

Let $S$ be the polynomial algebra $k[x_1, \dots, x_r]$ on indeterminates $\{x_{i}\}$, where $|x_{i}|=  -|z_{i}|-1$, again viewed as a DG algebra with zero differential; $\D(S)$ and $\Db(S)$ have the expected meaning. In this case, $\Db(S)=\thick_{S}(S)$, that is to say, every DG $S$-module with finitely generated homology is perfect; see \cite[Remark~7.5]{ABIM}.

The main result that we need is the following equivalence of categories from \cite[Theorem 7.4]{ABIM}. It is a DG analogue of the Bernstein-Gelfand-Gelfand correspondence.

\begin{thm} 
\label{abim74}
\pushQED{\qed}
There exists an exact functor $\sfh\col \D(\Lambda) \to \D(S)$ with the property that its restriction to $\Db(\Lambda)$ is an equivalence 
\[
\sfh\col  \Db(\Lambda) \xra{\ \equiv\ }\Db(S). \qedhere
\]
\end{thm}

The construction of the functor $\sfh$ is sketched below. We refer the reader to \cite[\S7]{ABIM} (note the Corrigendum) for a complete proof of the theorem. 

The graded dual $\Hom_k(S,k)$ is a then a DG $S$-module. Consider the algebra $\Lambda \otimes_k S$ as a DG algebra with zero differential. Let $F$ denote the DG module over it with underlying graded module $\Lambda \otimes_{k} \Hom_{k}(S,k)$ and  differential given by multiplication by $\delta = \sum_{i = 1}^r z_i \otimes x_i$. It is a straightforward exercise to show that $\delta^2 = 0$. 

The DG ($\Lambda \otimes_k S$)-module $F$ is the DG module $\Hom_{k}(X,k)$ constructed in \cite[\S 7]{ABIM}. The result below is thus contained in (7.6.4), (7.6.5) and (7.6.3) of \cite{ABIM}. The gist of (1) and (2) is that $F$ is a semi-free resolution of $k$, viewed as a DG module over $\Lambda$.

\begin{prop}
\label{prop:propF}
The following statement hold.
\begin{enumerate}
\item
The map $\varepsilon\col F \to k$ induced by  augmentations $\Lambda\to k$ and $\Hom_{k}(S,k)\to k$, is a morphism of DG $\Lambda$-modules and a quasi-isomorphism. 
\item $\Hom_{\Lambda}(F,-)$ preserves quasi-isomorphisms.
\item 
The map $S\to \Hom_{\Lambda}(F,F)$ induced by the multiplicative action of $S$ on $F$ is a morphism of DG $S$-modules, and a quasi-isomorphism. \qed
\end{enumerate}
\end{prop}

For any DG $\Lambda$-module $M$, the complex $\Hom_\Lambda(F,M)$ is a $\dg$ $S$-module. Hence, thanks to Proposition~\ref{prop:propF}(2), one can define an exact functor 
\[
\sfh\col \D(\Lambda) \to \D(S)\quad\text{with $\sfh(M) = \Hom_\Lambda(F,M)$}.
\]
The other parts of Proposition~\ref{prop:propF} imply that this induces an equivalence of triangulated categories 
between $\Db(\Lambda)$ and $\Db(S)$. From this we get a natural notion of support variety for DG $\Lambda$-modules. 

\begin{defi}
The support of $M\in\Db(\Lambda)$ is the subset of $\spec S$ defined by
\[
V_\Lambda(M) = \Supp^{*}_S\,\sfh(M) \subseteq \spec S\,.
\]
\end{defi}
From \eqref{eq:suppH} one then gets the first equality below:
\[
V_\Lambda(M) \ = \ \Supp^{*}_S\,\HHH(\sfh(M)) \ = \ \Supp^{*}_S\,\Ext_\Lambda(k,M)\,.
\]
The second one is from Proposition~\ref{prop:propF}. With this information, we can easily prove the main result of this section. 

\begin{thm}
\label{thm:hopkins-lambda}
The following statements hold:
\begin{enumerate}
\item
For  $M,N$ in $\Db(\Lambda)$ one has $V_\Lambda(M) \subseteq V_\Lambda(N)$ if and only if $M$ is in $\thick_\Lambda(N)$.
\item 
For any closed subset $V\subseteq \spec S$ there exists  $M\in\Db(\Lambda)$ with $V_\Lambda(M) = V$.
\end{enumerate}
\end{thm}

\begin{proof}
By hypothesis, $\Supp^{*}_S\,\sfh(M)\subseteq \Supp^{*}_S\,\sfh(N)$, so Theorem~\ref{thm:hopkins-dga} yields that $\sfh(M)$ is in $\thick_{S}(\sfh(N))$; as $\sfh$ is an equivalence it follows that $M$ is in  $\thick_\Lambda(N)$.

For the second statement, we observer that $V$ is the variety of a ideal $I \subset S$. So we let $M = \sfh^{-1}(S/I).$
\end{proof}

\begin{rem}
The results in this section carry over to the case when the generators $\{z_{i}\}$ of the exterior algebra $\Lambda$ are in degree zero; in this case $S$ will be a polynomial algebra on generators of degree $-1$. 
\end{rem}

\section{Artinian complete intersection rings}
\label{sec:ci}
Let $k$ be a field, $k[z_{1},\dots,z_{r}]$ the polynomial algebra over $k$ in indeterminates $\bs z = z_{1},\dots,z_{c}$, and let $f_{1},\dots,f_{r}$ be a regular sequence in $(\bs z)^{2}$. Set
\[
R= k[z_{1},\dots,z_{r}]/(f_{1},\dots,f_{r})
\]
Thus $R$ is a \emph{complete intersection ring} of (Krull) dimension zero. The main result in this section is Theorem~\ref{thm:hopkins-ci}, that is an analogue of Hopkins' theorem for $\Db(R)$, the bounded derived category of finitely generated $R$-modules.

Since $f_{j}$ is in $(\bs z)^{2}$ for each $1\leq j\leq r$, we can write 
\begin{equation}
\label{eqn:relations}
f_{j} = \sum_{1\leqslant h\leqslant i\leqslant r}c_{hi,j}\,z_{h}z_{i}
\end{equation}
with each $c_{hi,j}$ in $k[z_{1},\dots,z_{r}]$. These elements give rise to structure constants in the Ext-algebra of $R$, as discovered by Sj\"odin~\cite[Theorem 5]{Sj}, and described below. 

\begin{rem}
\label{rem:ext}
Let $R = k\{\xi_{1},\dots,\xi_{r},\theta_{1},\dots,\theta_{r}\}$ be the tensor algebra on $k$, where the $\{\xi_{i}\}$ 
and the $\{\theta_{j}\}$ are indeterminates of degree $1$ and $2$, respectively. Then $\Ext_{R}(k,k)$ is this tensor algebra modulo the ideal generated by the relations
\begin{gather*}
\xi_{h}\xi_{i} + \xi_{h}\xi_{i} = - \sum_{j=1}^{r}\overline{c}_{hi,j}\theta_{j} \quad\text{for $h<i$}\quad \text{and} \quad 
\xi_{h}^{2} = - \sum_{j=1}^{r}\overline{c}_{hh,j}\theta_{j} \\
\theta_{i}\xi_{j} - \xi_{j} \theta_{i} = 0 = \theta_{i}\theta_{j} - \theta_{j} \theta_{i} \quad \text{for all $i,j$}
\end{gather*}
where for any $c$ in $k[z_{1},\dots,z_{r}]$, we write $\overline{c}$ for its image in the polynomial ring $k=k[z_{1},\dots,z_{r}]/(z_{1},\dots,z_{r})$.

Thus, $k[\bs{\theta}]=k[\theta_{1},\dots,\theta_{r}]$ is a polynomial ring in the center of $\Ext_{R}(k,k)$, and $\Ext_{R}(k,k)$ is a finitely generated (even free) module over it, generated by $\xi_{1},\dots,\xi_{r}$.
\end{rem}

\begin{exm}
\label{exm:ke}
The example to bear in mind is the group algebra of an elementary abelian $p$-group, $E = (\bZ/p)^{r}$, over a field of characteristic $p$. Since
\[
kE\cong k[z_{1},\dots,z_{r}]/(z_{1}^{p},\dots,z_{r}^{p})
\]
one can choose $c_{hi,j}$ to be $z_{j}^{p-2}$ when $h=i=j$, and zero otherwise. Then the description of its Ext-algebra from Remark~\ref{rem:ext} recovers a well-known computation:
\[
\Ext_{kE}(k,k) =
\begin{cases}
k[\xi_{1},\dots,\xi_{r}] & \text{if $p=2$}\\
k\langle \xi_{1},\dots,\xi_{r}\rangle\otimes_{k} k[\theta_{1},\dots,\theta_{r}] & \text{for $p$ odd}.
\end{cases}
\]
\end{exm}


For any complex $M$ of $R$-modules, the graded $k$-vector space $\Ext_{R}(k,M)$ has a structure of a graded right module over $\Ext_{R}(k,k)$, and hence over $k[\bs\theta]$. The next result  is due to Gulliksen~\cite[Theorem 3.1]{Gu} for general complete intersection rings; we furnish an elementary proof, specific to our context: 

\begin{lemma}
\label{lem:thickk}
There is an equality $\Db(R)=\thick_{R}(k)$. In particular,  for each $M$ in $\Db(R)$, the $k[\bs \theta]$-module $\Ext_{R}(k,M)$ is noetherian.
\end{lemma}

\begin{proof}
Let $M$ be  in $\Db(R)$. Replacing it by a quasi-isomorphic complex, if necessary, we may assume $M$ is a bounded complex of finitely generated $R$-modules. It then has a finite filtration $M \supseteq (\bs z)M \supseteq (\bs z)^{2}M\supseteq \cdots 0$ by sub-complexes, with sub-quotients isomorphic to a complex of $k$-vector space, of finite rank over $k$. Such a complex is in $\thick_{k}(k)$ and hence in $\thick_{R}(k)$. It follows that $M$ itself is in $\thick_{R}(k)$. This establishes that $\Db(R)\subseteq\thick_{R}(k)$; the reverse inclusion is a tautology. 

For the second statement, one has only to note that if $N'\to N\to N''\to $ is an exact triangle of complexes of $R$-modules, and the $k[\bs\theta]$-modules $\Ext_{R}(k,N')$ and $\Ext_{R}(k,N'')$ are finitely generated, then so is $\Ext_{R}(k,N)$.
\end{proof}

This permits one to introduce a notion of support for complexes of $R$-modules.

\begin{defi}
\label{defi:civariety}
The support variety of any complex $M\in\Db(R)$ is the subset 
\[
V_{R}(M) = \Supp^{*}_{k[\bs{\theta}]}\Ext_{R}(k,M) \subseteq \spec k[\bs{\theta}]\,.
\]
\end{defi}
Note that $V_{R}(k) = \spec k[\bs{\theta}]$.

One has an exact analogue of Theorem~\ref{thm:hopkins-dga}:

\begin{thm}
\label{thm:hopkins-ci}
Let $M$ and $N$ be complexes in $\Db(R)$. If $V_R(M) \subseteq V_R(N)$, then $M$ is inÊ $\thick_R(N)$. 
\end{thm}

See Remark~\ref{rem:cihistory} for antecedents. The result is proved by a reduction to the case of DG modules over exterior algebras, Theorem~\ref{thm:hopkins-lambda}, via functors described below.

\begin{construction}
\label{con:koszul}
Let $K$ be the Koszul complex on the sequence $z_{1},\dots,z_{r}$, viewed as a DG algebra. It is thus an exterior algebra over $R$ on indeterminates $y_{1},\dots, y_{r}$, with each $y_{i}$ of degree $-1$, with differential determined by 
\[
d(z_{i})= 0\quad\text{and} \quad d(y_{i}) = z_{i}\,.
\]
Since $K$ is a finite free complex of $R$-modules, the assignment $M\mapsto K\otimes_{R}M$ induces an exact functor $\sft\col \Db(R)\to \Db(K)$; it is left adjoint to the restriction functor $\sfr\col \Db(K)\to \Db(R)$ induced by the morphism of DG algebras $R\to K$. 
\end{construction}

The proof of the result below only uses the fact that $R$ is an artinian local ring; the complete intersection property is not required. The gist of the statement is that the functors $\sft$ and $\sfr$, though not equivalences, come close to that.

\begin{prop}
\label{prop:thickk}
One has $\thick_{R}(M) = \thick_{R}(\sfr\sft M)$ for each $M$ in $\Db(R)$.
\end{prop}

\begin{proof}
It suffices to prove that $\thick_{R}(R) = \thick_{R}(K)$; the stated equality follows by applying $-\lotimes RM$. Since $K$ is a finite free complex, it is in $\thick_{R}(R)$, so it remains to verify that $R$ is in $\thick_{R}(K)$. This is a special case of Theorem~\ref{thm:hopkins-dga}, since $K$ and $R$ are both perfect complexes with the same support---namely, $\{(\bs z)\}$---but there is also an elementary argument, exploiting the structure of the Koszul complex; see \cite[Lemma~6.0.9]{Hovey/Palmieri/Strickland:1997}. This is precisely where the hypothesis that $R$ is zero dimensional is used. 
\end{proof}

Let $\Lambda$ be an exterior algebra over $k$ on $r$ generators $\xi_{i}$ in degree $1$, and view it as a
DG algebra with zero differential. 

\begin{lemma}
\label{lem:KtoLambda}
The assignment $\xi_{j}\mapsto \sum_{1\leqslant h\leqslant i\leqslant r}c_{hi,j}z_{h}y_{i}$ induces a morphism $\Lambda\to K$ of DG algebras that is a quasi-isomorphism.
\end{lemma}

\begin{proof}[Commentary in lieu of a proof]
The said assignment induces a map $\Phi\col \Lambda\to K$ of graded algebras, by the universal properties of exterior algebras; note that the graded algebra underlying $K$ is strictly graded-commutative, for that too is an exterior algebra.  Given \eqref{eqn:relations}, it is clear that $\sum_{h,i}c_{hi,j}z_{h}y_{i}$ is a cycle in $K$ for each $j$, so $\Phi$ is a morphism of DG algebras. One can complete the proof by invoking the fact that these cycles form a basis for the $k$-vectorspace $\HHH_{1}(K)$ and that $\HHH(K)$ is the exterior algebra on $\HHH_{1}(K)$; see \cite[Theorem~6]{Ta}.

For a proof specific to the case of Example~\ref{exm:ke}, see \cite[Lemma~7.1]{BIK3}.
\end{proof}

The quasi-isomorphism $\Lambda\to K$ of DG algebras from Lemma~\ref{lem:KtoLambda} induces an equivalence $\sfi\col \Db(K)\to \Db(\Lambda)$ of triangulated categories; see \cite[3.6]{ABIM}. This, along with the functors $\sft$ and $\sfr$ from Construction~\ref{con:koszul} makes for a bridge from $R$ to $\Lambda$:
\begin{equation}
\label{eq:bridge}
\xymatrixcolsep{2pc}
\xymatrix{ 
\Db(R) \ar@<+1ex>[rr]^-{\sft} && \Db(K) \ar@<+1ex>[ll]^-{\sfr} \ar@{->}[rr]_{\equiv}^{\sfi} && \Db(\Lambda)}
\end{equation}

The morphisms $R\to K\xla{\simeq}\Lambda$ of DG algebras induce a homomorphism of $k$-algebras $\Ext^{*}_{\Lambda}(k,k)\to \Ext^{*}_{R}(k,k)$. This map is one-to-one, and its image is precisely the subalgebra $k[\bs\theta]$. In this way, we identify $V_{R}(k)$ and $V_{\Lambda}(k)$. 

\begin{prop}
\label{prop:vr}
There is an equality $V_{R}(M) = V_{\Lambda}(\sfi\sft M)$ for each $M$ in $\Db(R)$.
\end{prop}

\begin{proof}
Adjunction gives an isomorphism $\Ext_{\Lambda}(k,\sfi\sft M)\cong \Ext_{R}(k,M)$ and this is compatible with the homomorphism of $k$-algebras $\Ext_{\Lambda}(k,k)\to \Ext_{R}(k,k)$.
\end{proof}

\begin{proof}[Proof of Theorem~\emph{\ref{thm:hopkins-ci}}] The hypothesis translates to  $V_{\Lambda}(\sfi\sft M)\subseteq V_{\Lambda}(\sfi\sft N)$, by Proposition~\ref{prop:vr}, so  Theorem~\ref{thm:hopkins-lambda} yields that $\sfi\sft M$ is in $\thick_{\Lambda}(\sfi\sft N)$. Since $\sfi$ is an equivalence, this gives that $\sft M$ is in $\thick_{K}(\sft N)$, and hence the inclusion below:
\[
\thick_{R}(M) = \thick_{R}(\sfr\sft M)\subseteq \thick_{R}(\sfr\sft N) = \thick_{R}(N)\,.
\]
The equalities are from Proposition~\ref{prop:thickk}. 
\end{proof}

\begin{rem}
\label{rem:cihistory}
Theorem~\ref{thm:hopkins-ci} carries over to general artinian complete intersection rings, and the proof is identical, except for one step, Lemma~\ref{lem:KtoLambda}: there may be no morphisms between the DGA algebras $K$ and $\Lambda$ for the ring $R$ may not contain its residue field as a subring. However, they are still linked by a chain quasi-isomorphisms (see, for example, \cite[Lemma~6.4]{AI}) and that is all that is needed.

A solution to the realizability problem for $\Db(R)$ is contained in Bergh~\cite[Corollary 2.3]{Be}; see also \cite[Theorem 7.8]{AI2}: For any closed subset $V\subseteq V_R(k)$ there exists a complex $M$ in $\Db(R)$ with the property that $V_R(M) = V$. For another proof, closer to the spirit of this paper, see Avramov and Jorgensen~\cite{AJ}. From these and Theorem~\ref{thm:hopkins-ci}, one obtains  a bijection between the thick subcategories of $\Db(R)$ and specialization closed subsets of $V_{R}(k)$, for any artinian complete intersection ring $R$; see the proof of Corollary~\ref{cor:thick-cdga}.

Stevenson~\cite[Theorem~10.5]{St} has established an analogous classification for any local complete intersection ring; see also \cite{Iy}.
\end{rem}

\section{Group algebras}
\label{sec:kg}
The central result in this section is an extension of Hopkins' Theorem to group algebras. Throughout we assume that $G$ is any finite group and that $k$ is a field of characteristic $p>0$ dividing $|G|$. Keeping in line with notation in previous sections, we write $\Db(kG)$ for the bounded derived category of finitely generated $kG$-modules. This is a tensor triangulated category, where the tensor product $M\otimes N$ of complexes $M$ and $N$ is the complex $M\otimes_{k} N$ with the diagonal $G$-action. 

Given a subgroup $E$ of $G$ and a complex $M$ of $kG$-modules, we write $M\da GE$ for the complex of $kE$-modules obtained from $M$ by restriction along the inclusion $kE\subseteq kG$. Given a complex $N$ of $kE$-modules, we write $N\ua EG$ for complex $kG\otimes_{kE}N$ of $kG$-modules induced from $N$.
 
Our main tool is the following result from \cite{C}, see also \cite[Theorem 8.2.7]{CTVZ}, that leads to a connection between thick subcategories of the bounded derived category of $kG$ with those of the group algebras of the elementary abelian $p$-subgroups of $G$. 

\begin{thm} 
\label{thm:filter}
There exists a $kG$-module $V$ and a filtration 
\[ 
\{0\} \ = M_0 \subseteq M_1 \subseteq \dots \subseteq M_t  \ = \ k \oplus V
\]
where for every $i= 1, \dots, t$, there is an elementary abelian $p$-subgroup $E_i \subseteq G$ and a finitely generated $kE_i$-module $W_i$ such that $M_i/M_{i-1} \cong (W_i)\ua {E_{i}}G$.  \qed
\end{thm}

In \cite[Corollary 2.4]{C} this result was used to prove that the stable category $\stmod(kG)$ is generated by modules induced from elementary abelian $p$-subgroups of $G$. The same is true for the bounded derived category.

\begin{cor} 
\label{gens-kg}
With $E_1, \dots, E_t$ the elementary abelian p-subgroups of $G$ from Theorem~\ref{thm:filter}, there is an equality
\[
\Db(kG) = \thick_{G}(\bigoplus_{i=1}^t k\ua{E_i}G)\,.
\]
\end{cor}

\begin{proof}
Any complex $C$ of $kG$-modules is a direct summand of $C \otimes (k \oplus V)$, which has a filtration by submodules $C \otimes M_i$ with quotients
\[
(C \otimes M_i)/(C \otimes M_{i-1}) \ \cong \ C \otimes (M_i/M_{i-1}) \ \cong \ 
C \otimes (W_i)\ua {E_{i}}G  \ \cong \ 
(C\da G{E_i}\otimes W_i ){\ua {E_i}G}\,,
\] 
the last isomorphism being Frobenius reciprocity. Thus $C$ is in the thick subcategory generated by the $(C\da G{E_i}\otimes W_i ){\ua {E_i}G}$.

It remains to note that when $C$ is in $\Db(G)$, the complex $C{\da G{E_i}}\otimes W_{i}$ is in $\Db(E_{i})$, and hence in $\thick_{E}(k)$; see Lemma~\ref{lem:thickk}.
\end{proof}

We recall a construction of cohomological varieties for complexes of $kG$-modules. 

The cohomology algebra $\HHH^{*}(G,k)$ of $G$, that is to say the $k$-algebra $\Ext^{*}_{kG}(k,k)$, is graded commutative and finitely generated, and hence noetherian. As usual, $\HHH^{\bullet}(G)$ denotes $\oplus_{i\geqslant 0}\HHH^{2i}(G,k)$ when $p$ is odd, and all of $\HHH^{*}(G,k)$ when $p=2$. In any case, $\HHH^{\bullet}(G)$ is a commutative noetherian subalgebra of $\HHH^{*}(G,k)$.

For any complex $M$ of $kG$-modules, there is a homomorphism of $k$-algebras 
\[
\HHH^{*}(G,k)\to\Ext^{*}_{G}(M,M)\,,
\]
so  $\Ext^{*}_{G}(M,M)$ is endowed with a structure of a module over $\HHH^{\bullet}(G)$; it is a finitely generated module when $M$ is in $\Db(kG)$; see, for example, \cite[Theorem 6.5.1]{CTVZ}. We note that the result in \cite{CTVZ} are stated for $kG$-modules, but the arguments are easily adapted to apply to any complex in $\Db(kG)$. Alternatively, one can deduce the result for complexes from the one for modules using the fact that any $M\in\Db(kG)$ is in the thick subcategory generated by its total homology module, $\HHH(M)$.

\begin{defi}
\label{defi:Gvariety}
The support variety of any complex $M\in\Db(kG)$ is the subset
\[
V_{G}(M) = \Supp_{\HHH^{\bullet}(G)} \Ext^{*}_{G}(M,M)\subseteq \spec\, \HHH^{\bullet}(G)\,.
\]
\end{defi}
Since $V_{G}(k)=\spec\, \HHH^{\bullet}(G)$ the support of $M$ may be viewed as a subset of $V_{G}(k)$.

\begin{rem}
We need support varieties for objects in $\Db(kG)$ and so have to work with the affine variety $\spec\, \HHH^{\bullet}(G,k)$. If one is interested only in $\stmod(kG)$, then one could consider instead the corresponding projective varieties.
\end{rem}

\begin{rem}
\label{rem:Evariety}
Let $E$ be an elementary abelian $p$-group $E$ of rank $r$; its group algebra is then isomorphic to the complete intersection $R=k[z_{1},\dots,z_{r}]/(z_{1}^{p},\dots,z_{r}^{p})$. For any $M$ in $\Db(kE)$, the two notions of support varieties of $M$ in Definitions~\ref{defi:Gvariety} and \ref{defi:civariety} are naturally isomorphic:
\[
V_{E}(M) \cong V_{R}(M)\,.
\]
Indeed, to begin with, since $E$ is a $p$-group, there is an equality of supports:
\[
V_{E}(M) = \Supp_{\HHH^{\bullet}(E)} \Ext^{*}_{E}(k,M)
\]
Now let $k[\bs\theta]$ be as in  Remark~\ref{rem:ext}. One has inclusions of $k$-subalgebras $k[\bs\theta]\subseteq \HHH^{\bullet}(E)$, and the $k[\bs\theta]$ action on $\Ext^{*}(k,M)$ factors through this inclusion. The induced map
\[
\Supp_{\HHH^{\bullet}(E)} \Ext^{*}_{E}(k,M) \xra{\ \cong\ } \Supp_{k[\bs\theta]}\Ext^{*}_{E}(k,M)=V_{R}(M)
\]
is an isomorphism; see \cite[Theorem~7.1]{Av}.
\end{rem}

A thick subcategory $\CatC$ of $\Db(kG)$ is \emph{tensor ideal} provided $M \otimes N$ is in $\CatC$ whenever one of $M$ or $N$ is in $\CatC$.  We write $\thick_{G}^\otimes(N)$ for the tensor ideal thick subcategory of $\Db(kG)$ generated by a complex $N$. 

If $\Db(kG) = \thick_{G}(k)$, then we have that $\thick_{G}^{\otimes}(M) = \thick(M)$ for any complex $M$ of $kG$-modules. This happens, for example, if $G$ is a $p$-group.

With this background we get a version of Hopkins' Theorem for group algebras.

\begin{thm}
\label{thm:hopkins-kg} 
Let $G$ be a finite group and $k$ a field of characteristic $p>0$. If $M,N$ are complexes in $\Db(kG)$ with $V_G(M) \subseteq V_G(N)$, then $M$ is in $\thick_{G}^\otimes(N)$.
\end{thm}

\begin{proof}
Let $E$ be an elementary abelian $p$-subgroup of $G$. The hypothesis then yields the inclusion below:
\[
V_{E}(M{\da G{E}}) = (\res_{G,E}^*)^{-1}\, V_G(M) \subseteq 
(\res_{G,E}^*)^{-1}\, V_G(N) = V_{E}(N{\da GE})
\]
The equalities are by the Subgroup Restriction Theorem~\cite[Theorem 9.6.2]{CTVZ}. Thus, keeping in mind Remark~\ref{rem:Evariety}, one gets from Theorem~\ref{thm:hopkins-ci} that $M{\da GE}$ is in $\thick_{E}(N{\da GE})$, and hence, by Frobenius reciprocity, also that  
\[
M \otimes k{\ua EG}\ \in \ \thick_{G}(N \otimes k{\ua EG})\subseteq \thick_{G}^{\otimes}(N)\,.
\]
Let now $E_1, \dots, E_t$ the elementary abelian $p$-subgroups of $G$ from Theorem~\ref{thm:filter}. The inclusion above then yields the second inclusion below:
\[
\thick_{G}(M)\subseteq \ \thick_{G}(\bigoplus_{i=1}^{t}(M \otimes k{\ua{E_{i}}G}))\subseteq \thick_{G}^{\otimes}(N)\,.
\]
while the first one is by Corollary \ref{gens-kg}. This completes the proof of the theorem.
\end{proof}

For any specialization closed subset of $V$ of $V_G(k)$, we write $\Db(kG)_{V}$ for the subcategory consisting of all $M$ in $\Db(kG)$ such that $V_G(M)\subseteq V$. Using the properties of support from \cite[Proposition 9.7.2]{CTVZ}, it is not difficult to see that $\Db(kG)_{V}$ is a thick tensor ideal subcategory of $\Db(kG)$. Arguing as in the proof of Corollary~\ref{cor:thick-cdga}, and using Theorem~\ref{thm:hopkins-kg} instead of Theorem~\ref{thm:hopkins-dga}, one get the following result.

\begin{cor} 
\label{cor:thick-db-kg}
If $\CatC$ is a tensor ideal thick subcategory of $\Db(kG)$, then there exist a specialization closed subset $V$ of $V_G(k)$ such that $\CatC = \Db(kG)_{V}$. \qed
\end{cor}

Now, for any closed subset $V\subseteq V_{G}(k)$, one can construct an object $M\in\Db(kG)$ with $V_{G}(M)=V$, either using either the tensor product theorem from \cite[Theorem 9.6.4]{CTVZ}, or the more elementary approach from \cite[\S5]{AI2}. Combined with Corollary~\ref{cor:thick-db-kg} this gives a bijection between tensor ideal thick subcategories of $\Db(kG)$ and specialization closed subsets of $V_{G}(k)$.

There is an equivalence $\Db(kG)/\thick(kG) \equiv \stmod(kG)$ of triangulated categories. From this and the corollary above, it is easy to deduce the characterization of the tensor ideal thick subcategories of $\stmod(kG)$ proved in \cite{BCR}.

\end{document}